\documentclass[12pt]{article}
\usepackage[top=1in, bottom=1in, left=1in, right=1in]{geometry}     

\usepackage{url,hyperref,lineno,microtype,subcaption}
\usepackage[onehalfspacing]{setspace}
\usepackage{enumerate}
\usepackage{graphicx}
\usepackage{bm,amssymb,amsmath,amsthm}
\usepackage{epstopdf}
\usepackage{verbatim}
\usepackage{multirow}
\usepackage{natbib}
\usepackage{textcomp}
\usepackage{tikz}
\usetikzlibrary{positioning}
\usepackage{color}

\newcommand{\Rmnum}[1]{\text{\MakeUppercase{\romannumeral #1}}}

\DeclareMathOperator{\diag}{diag}
\DeclareMathOperator*{\argmax}{argmax}
\DeclareMathOperator*{\argmin}{argmin}
\DeclareMathOperator{\var}{Var}

\newtheorem{lemma}{Lemma}
\newtheorem{theorem}{Theorem}
\newtheorem{corollary}{Corollary}

\newtheorem{remark}{Remark}
\newtheorem{assumption}{Assumption}

\newtheorem{definition}{Definition}

\newcommand{\ttt}{\bm \theta}
\newcommand{\w}{\bm w}
\newcommand{\Real}{\mathbb{R}}

\newcommand{\cS}{\mathcal{S}}
\newcommand{\cR}{\mathcal{R}}
\newcommand{\cE}{\mathcal{E}}
\newcommand{\cH}{\mathcal{H}}
\newcommand{\cV}{\mathcal{V}}
\newcommand{\sgn}{\text{sgn}}
\newcommand{\Ber}{\text{Ber}}

\title{Statistical Analysis of Multi-Relational Network Recovery}
\author{Zhi Wang, Xueying Tang and Jingchen Liu}
\date{}
\begin{document}

\maketitle

\begin{abstract}

	In this paper, we develop asymptotic theories for a class of latent variable models for large-scale multi-relational networks.
	In particular, we establish consistency results and asymptotic error bounds for the (penalized) maximum likelihood estimators when the size of the network tends to infinity.
	The basic technique is to develop a non-asymptotic error bound for the maximum likelihood estimators through large deviations analysis of random fields.
    We also show that these estimators are nearly optimal in terms of minimax risk. 
\end{abstract}

\section{Introduction}\label{sec:intro}
	A multi-relational network (MRN) describes   multiple relations among a set of entities simultaneously. 
	Our work on MRNs is mainly motivated by its applications to knowledge bases that are repositories of information. 
	Examples of knowledge bases include WordNet \citep{miller1995wordnet}, Unified Medical Language System \citep{mccray2003upper}, and Google Knowledge Graph (\url{https://developers.google.com/knowledge-graph}).
They have been used as the information source in many natural language processing tasks such as word-sense disambiguation and machine translation \citep{gabrilovich2009wikipedia, scott1999feature, ferrucci2010building}. 
	A knowledge base often includes knowledge on a large number of real-world objects or concepts. 
	When a knowledge base is characterized by MRN, the objects and concepts corresponds to nodes, and  knowledge types are relations. 
	Figure \ref{fig:example} provides an excerpt from an MRN in which ``Earth'', ``Sun'' and ``solar system'' are three  nodes. The knowledge about the orbiting patterns of celestial objects forms a relation ``orbit'', and the knowledge on classification of the objects forms another relation ``belong to'' in the MRN.

\begin{figure}[htb]
	\centering
	\includegraphics[width=7cm]{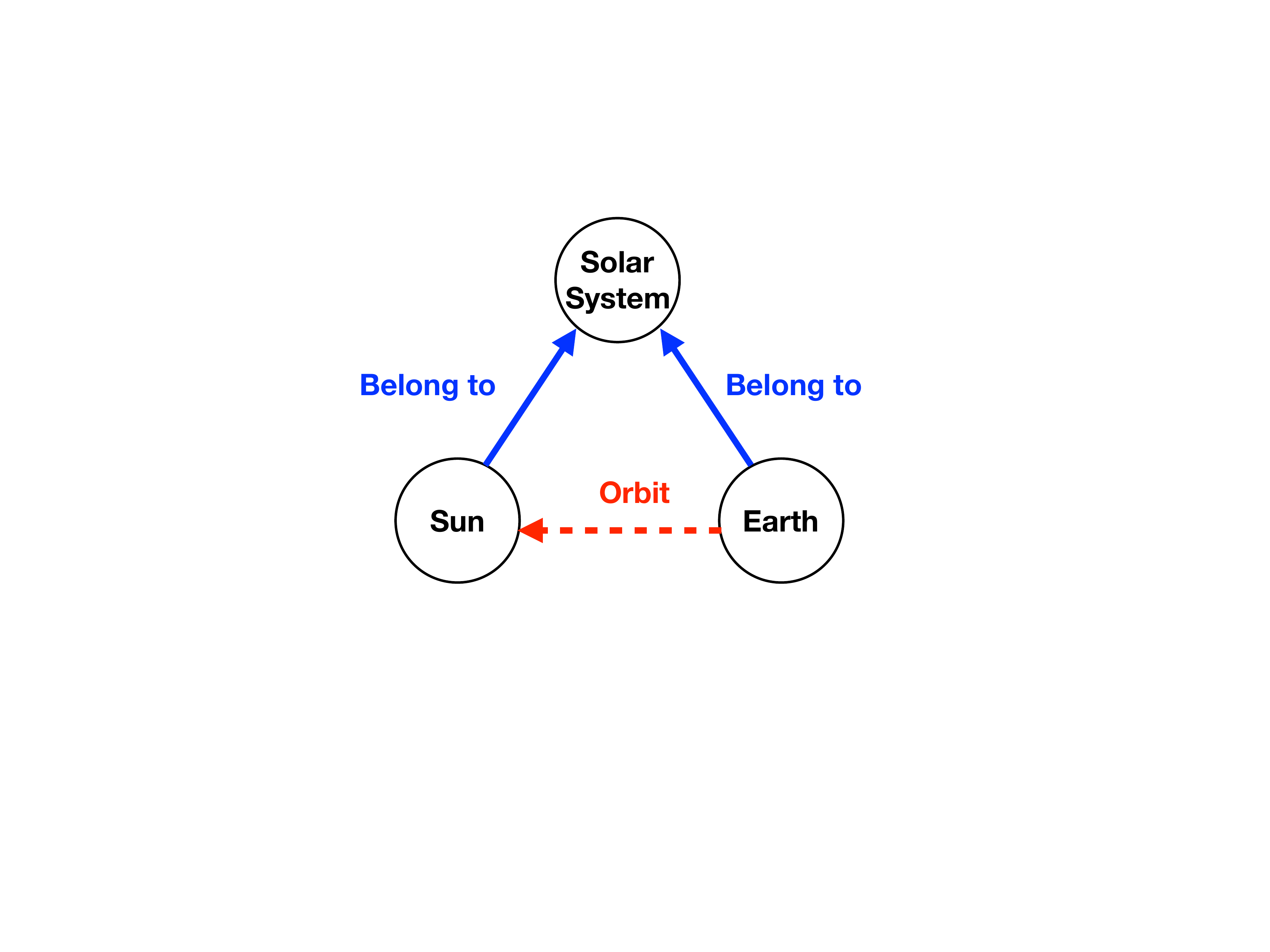}
	\caption{An example of the MRN representation of a knowledge base.}\label{fig:example}
\end{figure}

An important task of network analysis is to recover the unobserved network based on data. 
	In this paper, we consider a latent variable model for MRNs. 
	The presence of an edge from node $i$ to node $j$ of relation type $k$ is a Bernoulli random variable $Y_{ijk}$ with success probability $M_{ijk}$.
	Each node is associated with a vector, $\bm \theta$, called the embedding of the node. 
	The probability $M_{ijk}$ is modeled as a function $f$ of the embeddings, $\bm \theta_i$ and $\bm \theta_j$, and a relation-specific parameter vector $\bm w_k$.
	This is a natural generalization of the latent space model for single-relational networks \citep{hoff2002latent}.
Recently, it has been successfully applied to knowledge base analysis  \citep{bordes2013translation,wang2014knowledge,yang2015embedding,lin2015learning,garcia2016combining,trouillon2016complex,nickel2016holographic,liu2017analogical}. 
	Various forms of $f$ are proposed such as distance models \citep{bordes2013translation}, bilinear models \citep{trouillon2016complex,nickel2016holographic,liu2017analogical}, and neural networks \citep{socher2013reasoning}. 
	Computational algorithms are proposed to improve link prediction for knowledge bases \citep{kotnis2017analysis,kanojia2017enhancing}.
	The statistical properties of the embedding-based MRN models have not been rigorously studied. 
It remains unknown whether and to what extent the underlying distribution of MRN can be recovered, especially when there are a large number of nodes and relations.

	The results in this paper fill in the void by studying the error bounds and asymptotic behaviors of the estimators for  $M_{ijk}$'s for a general class of models.
	This is a challenging problem due to the following facts. 
	Traditional statistical inference of latent variable models often requires a (proper or improper) prior distribution for $\bm \theta_i$. 
	In such settings, one works with the marginalized likelihood with $\bm \theta_i$ integrated out.
	For the analysis of MRN, the sample size and the latent dimensions are often so large that the above-mentioned inference approaches are computationally infeasible. 
	For instance, a small-scale MRN could have a sample size as large as a few million, and the dimension of the embeddings is as large as several hundred.
	Therefore, in practice, the prior distribution is often dropped, and the latent variables $\bm \theta_i$'s are considered as additional parameters and estimated via maximizing the likelihood or penalized likelihood functions.
	The parameter space is thus substantially enlarged due to the addition of $\bm \theta_i$'s whose dimension is proportionate to the number of entities.
	As a result, in the asymptotic analysis, we face a double-asymptotic regime of both the sample size and the parameter dimension.

	In this paper, we develop results for the (penalized) maximum likelihood estimator of such models and show that under regularity conditions the estimator is consistent.
	In particular, we overcome the difficulty induced by the double-asymptotic regime via non-asymptotic bounds for the error probabilities. 
	Then, we show that the distribution of MRN can be consistently estimated in terms of average Kullback-Leibler (KL) divergence even when the latent dimension increases slowly as the sample size tends to infinity.
	A probability error bound is also provided together with the upper bound for the risk (expected KL divergence). 
	We further study the lower bound and show the near-optimality of the estimator in terms of minimax risk. 
	Besides the average KL divergence, similar results can be established for other criteria such as link prediction accuracy.

	The outline of the remaining sections is as follows. In Section \ref{sec:model}, we provide the model speicification and formulate the problem. Our main results are presented in Section \ref{sec:main}. 
	Finite sample performance is examined in Section \ref{sec:numerical_example} through simulated and real data examples. Concluding remarks are included in Section \ref{sec:discussion}.

\section{Problem setup}\label{sec:model}
\subsection{Notation}
Let $|\cdot|$ be the cardinality of a set and $\times$ be the Cartesian product. Set $\left\{1,\ldots, N\right\}$ is denoted by $[N]$. The sign function $\sgn(x)$ is defined to be $1$ for $x \geq 0$ and $0$ otherwise. The logistic function is denoted by $\sigma(x) = e^x / (1+e^x)$. Let $1_A$ be the indicator function on event $A$. We use $U\left[a,b\right]$ to denote the uniform distribution on $[a,b]$ and $\Ber(p)$ to denote the Bernoulli distribution with probability $p$. The KL divergence between $\Ber(p)$ and $\Ber(q)$ is written as $D(p||q) = p\log\frac{p}{q} + (1-p) \log \frac{1-p}{1-q}$. We use $\|\cdot\|$ to denote the Euclidean norm for vectors and the Frobenius norm for matrices.

For two real positive  sequences $\left\{a_n\right\}$ and $\left\{b_n\right\}$, we write $a_n = O(b_n)$ if $\limsup_{n \rightarrow \infty} a_n/b_n < \infty$. Similarly, we write $a_n = \Omega(b_n)$ if $\limsup_{n \rightarrow \infty}  b_n / a_n < \infty$ and   $a_n = o(b_n)$ if $\lim_{n \rightarrow \infty}  a_n / b_n = 0$. We denote $a_n \lesssim b_n$ if $\limsup_{n \rightarrow \infty}  a_n/b_n \leq 1$. When  $\left\{a_n\right\}$ and $\left\{b_n\right\}$ are negative sequences, $a_n \lesssim b_n$ means $\liminf_{n \rightarrow \infty}  a_n/b_n \geq 1$. In some places, we use $b_n \gtrsim a_n$ as an interchangeable notation of $a_n \lesssim b_n$. Finally,  if $\lim_{n \rightarrow \infty}  a_n / b_n = 1$, we write $a_n \sim b_n$. 

\subsection{Model}
	Consider an MRN with $N$ entities and $K$ relations. Given $i, j \in [N]$ and $k \in [K]$, the triple $\lambda = (i,j,k)$ corresponds to the edge from entity $i$ to entity $j$ of relation $k$. Let $\Lambda = [N] \times [N] \times [K]$ denote the set of all edges.
	We assume in this paper that an edge can be either present or absent in a network and use $Y_\lambda \in \{0, 1\}$ to indicate the presence of edge $\lambda$. In some scenarios, the status of an edge may have more than two types. Our analysis can be generalized to accommodate these cases.

We associate each entity $i$ with a vector $\ttt_i$ of dimension $d_E$ and each relation $k$ with a vector $\w_k$ of dimension $d_R$. Let $\cE \subseteq \Real^{d_E}$ be a compact domain where the embeddings $\ttt_1, \ldots, \ttt_N$ live. We call $\cE$ the entity space. Similarly, we define a compact relation space $\cR \subseteq \Real^{d_R}$ for the relation-specific parameters $\w_1, \ldots, \w_K$. Let $\bm x = \left(\ttt_1, \ldots, \ttt_N, \w_1, \ldots, \w_K\right)$ be a vector in the product space $\Theta = \cE^N \times \cR^K$. The parameters associated with edge $\lambda = (i, j, k)$ is then $\bm x_\lambda = (\ttt_i, \ttt_j, \w_k)$. We assume that given $\bm x$, elements in $\left\{Y_\lambda \mid \lambda \in \Lambda\right\}$ are independent with each other and that the log odds of $Y_\lambda = 1$ is 
\begin{equation}\label{eq:model}
\log \frac{P\left(Y_\lambda = 1 | \bm x \right)}{P\left(Y_\lambda = 0| \bm x\right)} = \phi\left(\bm x_\lambda\right), ~\text{for}~\lambda \in \Lambda.
\end{equation}
Here $\phi$ is defined on $\cE^2 \times \cR$, and $\phi\left(\bm x_\lambda\right)$ is often called the score of edge $\lambda$. 

We will use $Y$ to represent the $N \times N \times K$ tensor formed by $\left\{Y_\lambda \mid \lambda \in \Lambda\right\}$ and $M(\bm x)$ to represent the corresponding probability tensor $\left\{P(Y_\lambda = 1 \mid \bm x) \mid \lambda \in \Lambda\right\}$. Our model is given by
\begin{align}
Y_\lambda &\sim \Ber\left(M_\lambda\left(\bm x^*\right)\right),  \label{eq:model:Y} \\ 
M_\lambda(\bm x)  &= \sigma\left(\phi\left(\bm x_\lambda\right)\right), \lambda \in \Lambda, \label{eq:model:M}
\end{align}
where $\bm x^*$ stands for the true value of $\bm x$ and $Y_\lambda$'s are independent. In the above model, the probability of the presence of an edge is entirely determined by the embeddings of the corresponding entities and the relation-specific parameters. This imposes a low-dimensional latent structure on the probability tensor $M^* = M(\bm x^*)$. 

We specify our model using a generic function $\phi$. It includes various existing models as special cases. 
Below are two examples of $\phi$.
\begin{enumerate}
	\item Distance model \citep{bordes2013translation}. 
	\begin{equation}\label{eq:phi_1}
	\phi\left(\ttt_i, \ttt_j, \w_k\right) = b_k - \|\ttt_i + \bm a_k - \ttt_j\|^2,
	\end{equation}
	where $\ttt_i, \ttt_j, \bm a_k \in \mathbb{R}^{d}$, $b_k \in \mathbb{R}$ and $\bm w_k = (\bm a_k, b_k)$. In the distance model, relation $k$ from node $i$ to node $j$ is more likely to exist if $\ttt_i$ shifted by $\bm a_k$ is closer to $\bm \theta_j$ under the Euclidean norm.
		
	\item Bilinear model \citep{yang2015embedding}. 
	\begin{equation}\label{eq:phi_2}
	\phi\left(\ttt_i, \ttt_j, \w_k\right) = \ttt_i^T \diag(\w_k) \ttt_j,
	\end{equation}
	where $\ttt_i, \ttt_j, \w_k \in \mathbb{R}^d$ and $\diag(\w_k)$ is a diagonal matrix with $\w_k$ as the diagonal elements.
	Model \eqref{eq:phi_2} is a special case of the more general model $\phi\left(\ttt_i, \ttt_j, \w_k\right) = \ttt_i^T W_k \ttt_j$, where $W_k \in \mathbb{R}^{d \times d}$ is a matrix parametrized by $\w_k \in \mathbb{R}^{d_R}$. \citet{trouillon2016complex}, \citet{nickel2016holographic} and \citet{liu2017analogical} explored different ways of constructing $W_k$.
\end{enumerate}

Very often, only a small portion of the network is observed \citep{min2013distant}. 
We assume that each edge in the MRN is observed independently with probability $\gamma$ and that the observation of an edge is independent of $Y$. Let $\cS \subset \Lambda$ be the set of observed edges. Then the elements in $\cS$ are independent draws from $\Lambda$. For convenience, we use $n$ to represent the expected number of observed edges, namely, $n = E \left[|\cS| \right]= \gamma |\Lambda| = \gamma N^2 K$. Our goal is to recover the underlying probability tensor $M^*$ based on the observed edges $\{Y_\lambda \mid \lambda \in \cS\}$.

\begin{remark}
	Ideally, if there exists $\bm x^*$ such that $Y_\lambda = \sgn\left(M_\lambda(\bm x^*) - \frac{1}{2}\right)$ for all $\lambda \in \Lambda$, then $Y$ can be recovered with no error under $\bm x^*$. This is, however,  a rare case in practice, especially for large-scale MRN. A relaxed assumption is that $Y$ can be recovered with some low dimensional $\bm x^*$ and noise $\left\{\epsilon_\lambda\right\}$ such that 
	\begin{equation}\label{eq:model_noise}
	Y_\lambda = \sgn\left(M_\lambda(\bm x^*) + \epsilon_\lambda - \frac{1}{2}\right), \quad \epsilon_\lambda \overset{i.i.d}{\sim} U\left[-\frac{1}{2}, \frac{1}{2}\right], \quad \forall \lambda \in \Lambda.
	\end{equation}
	By introducing the noise term, we formulate the deterministic MRN as a random graph. 
	The model described in \eqref{eq:model:Y} is an equivalent but simpler form of \eqref{eq:model_noise}.
\end{remark}

\subsection{Estimation}\label{sec:estimation}

According to \eqref{eq:model:Y}, the log-likelihood function of our model is 
\begin{equation}\label{eq:llh}
l\left(\bm x; Y_{\cS}\right) = \sum_{\lambda \in \cS} Y_{\lambda} \log M_\lambda(\bm x) + \left(1- Y_\lambda\right)  \log \left(1-  M_\lambda(\bm x)\right).
\end{equation}
We omit the terms $\sum_{\lambda \in \mathcal{S}} \log \gamma + \sum_{\lambda \notin \mathcal{S}} \log\left(1-\gamma\right)$ in \eqref{eq:llh} since $\gamma$ is not the parameter of interest.
To obtain an estimator of $M^*$, we take the following steps. 
\begin{enumerate}
  	\item ~Obtain the maximum likelihood estimator (MLE) of $\bm x^*$,
  	\begin{equation}\label{eq:mle}
    \hat{\bm x} = \argmax_{\bm x\in \Theta} l \left(\bm x; Y_{\mathcal{S}} \right).
  	\end{equation}
  	\item ~Use the plug-in estimator  
  	\begin{equation}\label{Mhat}
    \hat{M} = M(\hat{\bm x})
  	\end{equation}
	as an estimator of $M^*$.
\end{enumerate}
In \eqref{eq:mle}, the estimator $\hat{\bm x}$ is a maximizer over the compact parameter space $\Theta = \cE^N \times \cR^K$. The dimension of $\Theta$ is $$m = Nd_E + Kd_R,$$ which grows linearly in the number of entities $N$ and the number of relations $K$.

\subsection{Evaluation criteria}

 We consider the following criteria to measure the error of the above-mentioned estimator. They will be used in both the main results and numerical studies.
 \begin{enumerate}[(a)]
 	\item ~ Average KL divergence of the predictive distribution from the true distribution
 	\begin{equation}\label{eq:loss_function}
 	L(\hat M, M^*) =  \frac{1}{|\Lambda|} \sum_{\lambda \in \Lambda} D(M_\lambda^* || \hat M_\lambda).
 	\end{equation}
 	\item ~ Mean squared error of the predicted scores
 	\begin{equation}\label{eq:mse_loss}
 	MSE_\phi = \frac{1}{|\Lambda|} \sum_{\lambda \in \Lambda} \left(\phi(\hat{\bm x}_\lambda) - \phi(\bm x_\lambda^*) \right)^2.
 	\end{equation}
 	\item ~ Link prediction error
 	\begin{equation}\label{eq:estimation_error_loss}
 	\widehat{err} = \frac{1}{|\Lambda|} \sum_{\lambda \in \Lambda} 1_{\hat Y_\lambda \neq  Y_\lambda^*},
 	\end{equation}
 	where  $\hat Y_\lambda = \sgn\left(\hat M_\lambda - \frac{1}{2}\right)$ and $Y_\lambda^* = \sgn\left(M_\lambda^* - \frac{1}{2}\right)$.
 \end{enumerate}

 \begin{remark}
  	The latent attributes of entities and relations are often not identifiable, so the MLE $\hat{\bm x}$ is not unique. For instance, in \eqref{eq:phi_1}, the values of $\phi$ and $M(\bm x)$ remain the same if we replace $\bm \theta_i$ and $\bm a_k$ respectively by $\Gamma \bm \theta_i + \bm t$ and $\Gamma \bm a_k$, where $\bm t$ is an arbitrary vector in $\mathbb R^{d_E}$ and $\Gamma$ is an orthonormal matrix.
	Therefore, we consider the mean squared error of scores, which are identifiable.
 \end{remark}



\section{Main Results}\label{sec:main}
We first provide results of the MLE in terms of KL divergence between the estimated and the true model. 
Specifically, we investigate the tail probability $P(L(\hat M, M^*) > t)$ and the expected loss $E[L(\hat M, M^*)]$. In Section \ref{sec:upper}, we discuss upper bounds for the two quantities. The lower bounds are provided in Section \ref{sec:lower}. In Section \ref{sec:extensions}, we extend the results to penalized maximum likelihood estimators (pMLE) and other loss functions. All proofs are deferred to the Appendix.

\subsection{Upper bounds}\label{sec:upper}
	We first present an upper bound for the tail probability $P(L(\hat M, M^*) > t)$ in Lemma \ref{lemma:upper:tail}. The result depends on the tensor size, the number of observed edges, the functional form of $\phi$, and the geometry of parameter space $\Theta$.
	The lemma explicitly quantifying the impact of these element on the error probability. 
	It is key to the subsequent analyses. 
	Lemma \ref{lemma:upper:risk} gives a non-asymptotic upper bound for the expected loss (risk). We then establish the consistency of $\hat M$ and the asymptotic error bounds in Theorem $\ref{thm:upper}$.

We will make the following assumptions throughout this section. 
\begin{assumption}\label{assumption:star}
	 $\bm x^* \in \Theta = \cE^N \times \cR^K$, where $\cE$ and $\cR$ are Euclidean balls of radius $U$.
\end{assumption}
\begin{assumption}\label{assumption:lip}
    The function $\phi$ is Lipschitz continuous under the Euclidean norm, 
	\begin{equation}
	\left| \phi\left(\bm u\right) - \phi\left(\bm v\right) \right| \leq \alpha \| \bm u - \bm v \|, \quad \forall \bm u, \bm v \in \mathcal{E}^2 \times \mathcal{R}, 
	\end{equation}
	where $\alpha$ is a Lipschitz constant.
\end{assumption}

	Assumption \ref{assumption:star} is imposed for technical convenience. The results can be easily extended to general compact parameter spaces. 
	Let $C = \sup\limits_{\bm u \in \mathcal{E}^2 \times \mathcal{R}} | \phi(\bm u) |$. Without loss of generality, we assume that $C \geq 2$. 
	
\begin{lemma}\label{lemma:upper:tail}
	 Consider  $\hat M$ defined in \eqref{Mhat} and the average KL divergence $L$ in \eqref{eq:loss_function}. Under Assumptions \ref{assumption:star} and \ref{assumption:lip}, for every $t>0$, $\beta>0$ and  $0 < s < nt$,
	\begin{equation}\label{eq:lemma:upper:tail}
	P\left(L(\hat M, M^*) \geq  t\right) \leq   \exp\left\{ -\frac{nt-s}{C} h\left(\frac{1}{2} - \frac{s}{2nt}\right) \right\} \left(1+\frac{2\sqrt{3}\alpha Un(1+\beta)}{s}\right)^m + \exp\left\{-n\beta h(\beta)\right\},
	\end{equation}
	where $m = Nd_E + Kd_R$ is the dimension of $\Theta$, $n = \gamma N^2K$ is the expected number of observations, and $h(u) = (1+\frac{1}{u}) \log(1+u)-1$.
\end{lemma}

	In the proof of Lemma \ref{lemma:upper:tail}, we use Bennett's inequality to develop a uniform bound that does not depend on the true parameters. It is sufficient for the current analysis. If the readers need sharper bounds, they can read through the proof and replace the Bennett's bound by the usual large deviation rate function which provides a sharp exponential bound that depends on the true parameters. We don't pursue this direction in this paper.


Lemma \ref{lemma:upper:risk} below gives an upper bound of risk $E[L(\hat M, M^*)]$, which follows from Lemma \ref{lemma:upper:tail}.
\begin{lemma}\label{lemma:upper:risk}
	Consider $\hat M$ defined in \eqref{Mhat} and loss function $L$ in \eqref{eq:loss_function}. Let $C_1 = 18C$, $C_2 = 8\sqrt{3}\alpha U$ and $C_3 = 2\max\left\{C_1, C_2\right\}$. If  Assumptions \ref{assumption:star} and \ref{assumption:lip} hold and $\frac{n}{m} \geq C_2 + e $, then
	\begin{equation}\label{eq:lemma:upper:risk}
	E[L(\hat M, M^*)] \leq C_3 \frac{m}{n} \log \frac{n}{m} + \frac{C_1}{n}\exp\left\{- m \log \frac{n}{m}\right\} + \frac{3}{n}\exp\left\{-\frac{1}{3}\left( n+ C_3 m \log \frac{n}{m}\right)\right\}. 
	\end{equation}
\end{lemma}

We are interested in the asymptotic behavior of the tail probability in two scenarios: (i) $t$ is a fixed constant and (ii) $t$ decays to zero as the number of entities $N$ tends to infinity.
	The following theorem gives an asymptotic upper bound for the tail probability and the risk.
\begin{theorem}\label{thm:upper}
 Consider $\hat M$ defined in \eqref{Mhat} and the loss function $L$ in \eqref{eq:loss_function}. Let the number of entities $N \rightarrow \infty$ and $C, K, U,  d_E, d_R, \alpha$, and $\gamma$ be fixed constants. If Assumptions \ref{assumption:star} and \ref{assumption:lip} hold, we have the following asymptotic inequalities.\\
When $t$ is a fixed constant,
\begin{equation}\label{eq:thm:upper:constant}
\log P(L(\hat M, M^*) \geq t) \lesssim - \frac{t}{5C}n.
\end{equation}
When $t = 10C \frac{m}{n} \log \frac{n}{m}$,
\begin{equation}\label{eq:thm:upper:decay}
\log P(L(\hat M, M^*) \geq t) \lesssim - m \log \frac{n}{m}.
\end{equation}
Furthermore,
\begin{equation}\label{eq:thm:upper:risk}
E[L(\hat M, M^*)] \lesssim  10C \frac{m}{n} \log \frac{n}{m}.
\end{equation}
\end{theorem}

	The consistency of $\hat M$ is implied by \eqref{eq:thm:upper:constant} and the rate of convergence is ${| \log P(L(\hat M, M^*) \geq t)| }= \Omega(N^2)$ if $t$ is a fixed constant.
	The rate decreases to $\Omega(N \log N)$ for the choice of $t$ producing \eqref{eq:thm:upper:decay}. It is also implied by \eqref{eq:thm:upper:decay} that $L(\hat M, M^*) = O(\frac{1}{N} \log N)$ with high probability. 
	We show in the next section that this upper bound is reasonably sharp.	

	The condition that $K, U, d_E, d_R$, and $\alpha$ are fixed constants can be relaxed. For instance, we can let $U$, $d_E$, $d_R$, and $\alpha$ go to infinity slowly at the rate $O(\log N)$ and $K$  at the rate $O(N)$. We can let $\gamma$ go to zero provided that $\frac{m}{n} \log \frac{n}{m} = o(1)$.

\subsection{Lower bounds}\label{sec:lower}
We show in Theorem \ref{thm:lower} that the order of the minimax risk is $\Omega(\frac{m}{n})$, which implies the near optimality of  $\hat M$ in $\eqref{Mhat}$ and the upper bound $O(\frac{m}{n}\log \frac{n}{m})$ in Theorem \ref{thm:upper}. To begin with, we introduce the following definition and assumption. 
 
\begin{definition}
	For $\bm u  = (\ttt, \ttt', \w) \in \cE^2 \times \cR$, the $r$-neighborhood of $\bm u$ is 
	$$\mathcal{N}_r(\bm u) = \left\{(\bm \eta, \bm \eta', \bm \zeta) \in \cE^2 \times \cR \mid \|\bm \eta - \ttt \| \leq r, \|\bm \eta' - \ttt' \| \leq r, \|\bm \zeta - \w\| \leq r \right\}.$$
	Similarly, for $\bm x = (\ttt_1, \ldots, \ttt_N, \w_1, \ldots, \w_K) \in \cE^N \times \cR^K$, the $r$-neighborhood of $\bm x$ is 
	$$\mathcal{N}_r(\bm x) = \left\{(\bm \eta_1, \ldots, \bm \eta_N, \bm \zeta_1, \ldots, \bm \zeta_K) \in \cE^N \times \cR^K \mid \| \bm \eta_i - \ttt_i \| \leq r, \| \bm \zeta_k - \w_k \| \leq r, \forall i \in [N], k \in [K]\right\}.$$
\end{definition}

\begin{assumption}\label{assumption:non_constant}
	There exists $\bm u_0 \in \mathcal{E}^2 \times \mathcal{R}$ and $r, \kappa > 0$ such that $\mathcal{N}_r(\bm u_0) \subset \cE^2 \times \cR$ and 
	\begin{equation}
	\left| \sigma\left(\phi(\bm u)\right) - \sigma\left(\phi(\boldsymbol{v})\right) \right| \geq \kappa \| \bm u - \boldsymbol{v} \|, \quad \forall \bm u, \bm v \in \mathcal{N}_r(\bm u_0).
	\end{equation}
\end{assumption}

\begin{theorem} \label{thm:lower}
	Let $b=\sup_{\bm u \in \mathcal{N}_r(\bm u_0)} \sigma\left(\phi(\bm u)\right)$. Under Assumptions \ref{assumption:lip} and \ref{assumption:non_constant}, if $r^2 \geq \frac{(m/16-1) b(1-b)}{12\alpha^2n}$, then for any estimator $\hat{M}$, there exists $\bm x^* \in \Theta$ such that
	\begin{equation}\label{eq: minimax}
	P\left(L(\hat M, M^*) >  \tilde{C} \frac{m/16 - 1}{n}\right) \geq \frac{1}{2},
	\end{equation}
	where $\tilde{C}  = \frac{\kappa^2 b(1-b)}{108\alpha^2}$.
	Consequently, the minimax risk
	\begin{equation}
	\min\limits_{\hat M} \max\limits_{M^*} E[L(\hat M, M^*)]  \geq  \tilde{C} \frac{m/16 - 1}{2n}.
	\end{equation}
\end{theorem}

\subsection{Extensions}\label{sec:extensions}
\subsubsection{Reguralization}
In this section, we extend our asymptotic results in Theorem \ref{thm:upper} to regularized estimators. In practice, regularization is often considered to prevent overfitting. 
We consider a regularization similar to elastic net \citep{zou2005regularization}
\begin{equation}\label{eq:pen_loglik}
l_\rho\left(\bm x; Y_\cS\right) = l(\bm x; Y_\cS) - \rho_1 \| \bm x \|_1 - \rho_2 \| \bm x \|^2,
\end{equation}
where $\| \cdot \|_1$ stands for $L_1$ norm and $\rho_1, \rho_2 \geq 0$ are regularization parameters. 
The pMLE is
\begin{equation}\label{eq:pen_mle}
\hat{\bm x} = \argmax_{\bm x \in \Theta}  l_\rho(\bm x; Y_\cS).
\end{equation}
Note that the MLE in \eqref{eq:mle} is a special case of the pMLE above with $\rho_1 = \rho_2 = 0$. Since $\hat{\bm x}$ is shrunk towards $\bm 0$, without loss of generality, we assume that $\cE$ and $\cR$ are centered at $\bm 0$.
We generalize Theorem \ref{thm:upper} to pMLE in the following theorem.
\begin{theorem}\label{thm:extension:pen}
	Consider the estimator $\hat M$ given by \eqref{eq:pen_mle} and \eqref{Mhat} and the loss function $L$ in \eqref{eq:loss_function}.  Let the number of entities $N \rightarrow \infty$ and $C, K, U,  d_E, d_R, \alpha, \gamma$ be absolute constants.  If Assumptions \ref{assumption:star} and \ref{assumption:lip} hold and $\rho_1 + \rho_2 = o(\log N)$, then asymptotic inequalities \eqref{eq:thm:upper:constant}, \eqref{eq:thm:upper:decay}, and \eqref{eq:thm:upper:risk} in Theorem \ref{thm:upper} hold.
\end{theorem}

\subsubsection{Other loss functions}
We present some  results for the mean squared error loss $MSE_\phi$ defined in \eqref{eq:mse_loss} and the link prediction error $\widehat{err}$ defined in \eqref{eq:estimation_error_loss}. Corollaries \ref{cor:mse:upper} and \ref{cor:mse:lower} give upper and lower bounds for $MSE_\phi$, and Corollary $\ref{cor:err:upper}$ gives an upper bound for $\widehat{err}$ under an additional assumption.
\begin{corollary}\label{cor:mse:upper}
	Under the setting of Theorem \ref{thm:extension:pen} with the loss function replaced by $MSE_\phi$, we have the following asymptotic results.\\
	If $t$ is a fixed constant,
	\begin{equation}\label{eq:cor:mse:upper:constant}
	\log P\left(MSE_\phi \geq t\right) \lesssim - \frac{ 5 \sigma(C)\left(1-\sigma(C)\right) t}{2C}n.
	\end{equation}
	If $t = \frac{20C}{\sigma(C)\left(1-\sigma(C)\right)} \frac{m}{n} \log \frac{n}{m}$,
	\begin{equation}\label{eq:cor:mse:upper:decay}
	\log P\left(MSE_\phi\geq t\right) \lesssim - m \log \frac{n}{m}.
	\end{equation}
	Furthermore,
	\begin{equation}\label{eq:cor:mse:upper:risk}
	E\left[MSE_\phi\right] \lesssim \frac{20C}{\sigma(C)\left(1-\sigma(C)\right)} \frac{m}{n} \log \frac{n}{m}.
	\end{equation}
\end{corollary}

\begin{corollary}\label{cor:mse:lower}
	Under the setting of Theorem \ref{thm:lower} with the loss function replaced by $MSE_\phi$, we have 
	\begin{equation}\label{eq:cor:mse:minimax}
	P\left( MSE_\phi >  \tilde{C} \frac{m/16 - 1}{8n}\right) \geq \frac{1}{2},
	\end{equation}
	and
	\begin{equation}
	\min\limits_{\hat M} \max\limits_{M^*} E\left[MSE_\phi\right]  \geq  \tilde{C} \frac{m/16 - 1}{16n}.
	\end{equation}
\end{corollary}

\begin{assumption}\label{assumption:distance}
	There exists $\varepsilon > 0$ such that $\left| M_\lambda^* - \frac{1}{2} \right| \geq \varepsilon$ for every $\lambda \in \Lambda$.
\end{assumption}

\begin{corollary}\label{cor:err:upper}
		Under the setting of Theorem \ref{thm:extension:pen} with the loss function replaced by $\widehat{err}$ and Assumption \ref{assumption:distance} added, we have the following asymptotic results.\\
	    If $t$ is a fixed constant,
	    \begin{equation}\label{eq:cor:err:upper:constant}
	    \log P\left(\widehat{err} \geq t\right) \lesssim - \frac{2\varepsilon^2t}{5C}n.
	    \end{equation}
	    If $t = \frac{5C}{\varepsilon^2} \frac{m}{n} \log \frac{n}{m}$,
	    \begin{equation}\label{eq:cor:err:upper:decay}
	    \log P\left(\widehat{err}\geq t\right) \lesssim - m \log \frac{n}{m}.
	    \end{equation}
	    Furthermore,
	    \begin{equation}\label{eq:cor:err:upper:risk}
	    E\left[\widehat{err}\right] \lesssim \frac{5C}{\varepsilon^2} \frac{m}{n} \log \frac{n}{m}.
	    \end{equation}
\end{corollary}

\subsubsection{Sparse representations}

We are interested in sparse entity embeddings and relation parameters. Let  $\| \cdot \|_0$ be the number of non-zero elements of a vector and $\tau$ be a prespecified sparsity level of $\bm x$ (i.e. the proportion of nonzero elements). Let $m_\tau = m \tau$ be the upper bound of non-zero parameters, that is, $\| \bm x^* \|_0 \leq m_\tau$. Consider the following estimator
\begin{equation}\label{eq:sparse_mle}
\hat{\bm x} = \argmax_{\bm x \in \Theta} l\left(\bm x; \bm Y_\cS\right)  \quad  \text{subject to}  \quad \| \bm x\|_0 \leq m_\tau.
\end{equation}
The estimator defined above maximizes the $L_0$-penalized log-likelihood.

\begin{theorem}\label{thm:sparse}
	Consider  $\hat M$ defined in \eqref{eq:sparse_mle} and \eqref{Mhat} and the loss function $L$ in \eqref{eq:loss_function}.  Let the number of entities $N \rightarrow \infty$ and $\tau, C, K, U,  d_E, d_R, \alpha$ be absolute constants. Under Assumptions \ref{assumption:star} and \ref{assumption:lip}, the following asymptotic inequalities hold. \\
	If $t$ is a fixed constant,
	\begin{equation}\label{eq:thm:sparse:constant}
	\log P(L(\hat M, M^*) \geq t) \lesssim - \frac{t}{5C}n.
	\end{equation}
	If $t = 10C \frac{m_\tau}{n} \log \frac{n}{m_\tau}$,
	\begin{equation}\label{eq:thm:sparse:decay}
	\log P(L(\hat M, M^*) \geq t) \lesssim - m_\tau \log \frac{n}{m_\tau}.
	\end{equation}
	Furthermore,
	\begin{equation}\label{eq:thm:sparse:risk}
	E[L(\hat M, M^*)] \lesssim 10C \frac{m_\tau}{n} \log \frac{n}{m_\tau}.
	\end{equation}
\end{theorem}
We omit the results for other loss functions as well as the lower bounds since they can be analogously obtained.

\section{Numerical Examples}\label{sec:numerical_example}

In this section, we demonstrate the finite sample performance of $\hat M$ through simulated and real data examples.
Throughout the numerical experiments, AdaGrad algorithm \citep{duchi2011adaptive} is used to compute $\hat {\bm x}$ in \eqref{eq:mle} or \eqref{eq:pen_mle}. It is a first-order optimization method that combines stochastic gradient descent (SGD) \citep{robbins1951sa} with adaptive step sizes for finding the local optima. Since the objective function in \eqref{eq:mle} is non-convex, a global maximizer is not guaranteed. Our objective function usually has many global maximizers, but, empirically, we found the algorithm works well on MRN recovery and the recovery performance is insensitive to the choice of the starting point of SGD. Computationally, SGD is also more appealing to handle large-scale MRNs than those more expensive global optimization methods.

\vspace{0.6cm}

\subsection{Simulated Examples}\label{sec:simulation}
In the simulated examples, we fix $K=20$, $d_E= 20$ and consider various choices of $N$ ranging from 100 to 10,000 to investigate the estimation performance as $N$ grows. 
The function $\phi$ we consider is a combination of the distance model \eqref{eq:phi_1} and the bilinear model \eqref{eq:phi_2},
\begin{equation}\label{eq:phi_3}
\phi\left(\ttt_i, \ttt_j, \bm w_k\right) = \left(\ttt_i + \bm a_k - \ttt_j\right)^T\diag\left(\bm b_k\right)\left(\ttt_i + \bm a_k - \ttt_j\right),
\end{equation} 
where $\ttt_i, \ttt_j, \bm a_k, \bm b_k \in \mathbb{R}^d$ and $\w_k = (\bm a_k, \bm b_k)$.
We independently generate the elements of $\bm \ttt_i^*$, $\bm a_k^*$, and $\bm b_k^*$ from normal distributions $N(0,1)$, $N(0,1)$, and $N(0, 0.25)$, respectively, where $N(\mu, \sigma^2)$ denotes the normal distribution with mean $\mu$ and variance $\sigma^2$. To guarantee that the parameters are from a compact set, the normal distributions are truncated to the interval [-20, 20]. Given the latent attributes, each $Y_{ijk}$ is generated from the Bernoulli distribution with success probability $M_{ijk}^* = \sigma(\phi(\bm \ttt_i^*, \bm \ttt_j^*, \bm \w_k^*))$. The observation probability $\gamma$ takes value from $\{0.005, 0.01, 0.02\}$. For each combination of $\gamma$ and $N$, $100$ independent datasets are generated.
For each dataset, we compute $\hat{\bm x}$ and $\hat M$ in \eqref{eq:mle} and \eqref{Mhat} with AdaGrad algorithm and then calculate $L(\hat M, M^*)$ defined in \eqref{eq:loss_function} as well as the link prediction error $\widehat{err}$ defined in \eqref{eq:estimation_error_loss}. The two types of losses are averaged over the 100 datasets for each combination of $N$ and $\gamma$ to approximate the theoretical risks $E[L(\hat M, M^*)]$ and $E[\widehat{err}]$. These quantities are plotted against $N$ in log scale in Figure \ref{fig:sim_nobs}. As the figure shows, in general, both risks decrease as $N$ increases. When $N$ is small, $n/m$ is not large enough to satisfy the condition $n/m \geq C_2 + e$ in Lemma \ref{lemma:upper:risk} and the expected KL risk increases at the beginning. After $N$ gets sufficiently large, the trend agrees with our asymptotic analysis.

\begin{figure}[htb]
	\centering
	\includegraphics[width=\textwidth]{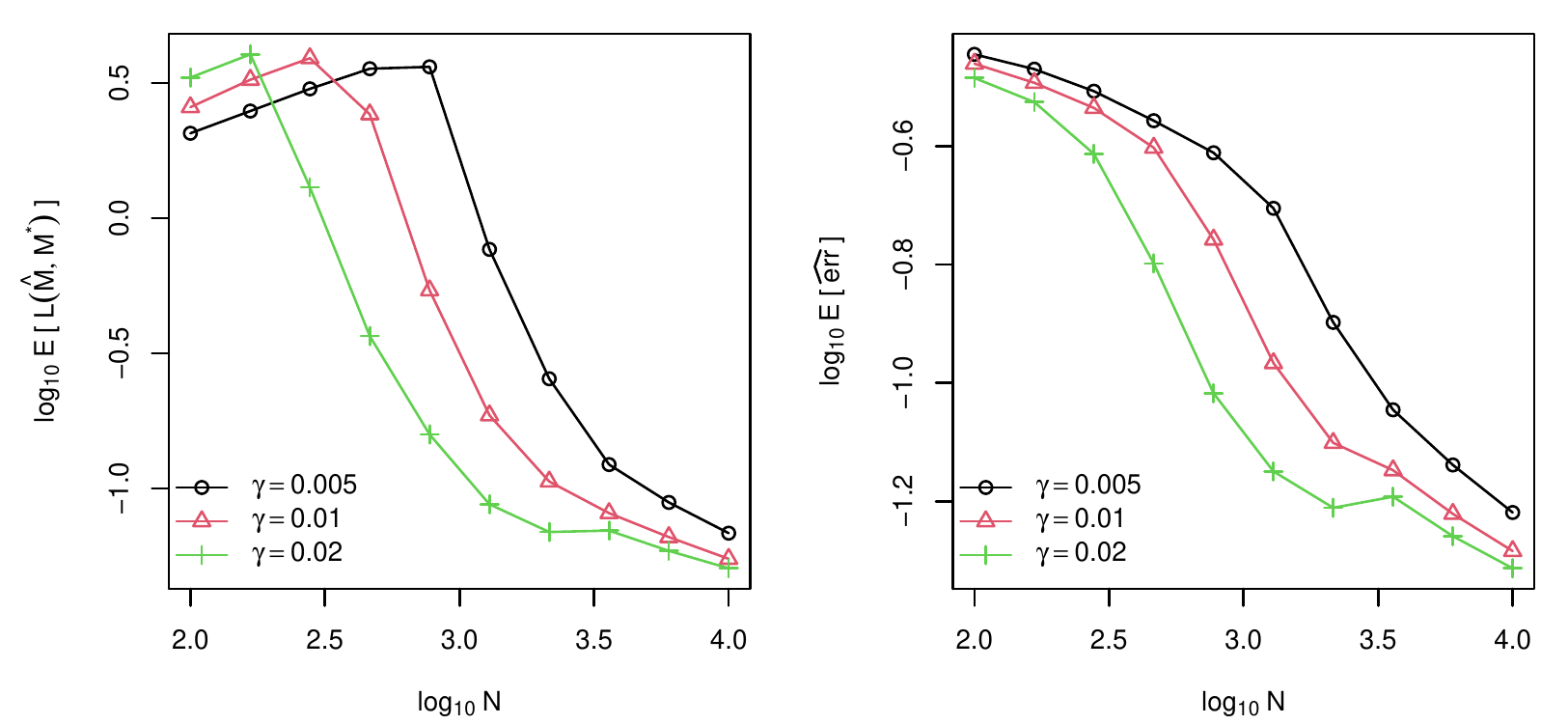}
	\caption{Average Kullback-Leibler divergence (left) and average link prediction error (right) of $\hat{M}$ for different choices of $N$ and $\gamma$.}\label{fig:sim_nobs}
\end{figure}

\subsection{Real data example: knowledge base completion}\label{sec:analysis}
WordNet \citep{miller1995wordnet} is a large lexical knowledge base for English. It has been used in word sense disambiguation, text classification, question answering, and many other tasks in natural language processing \citep{gabrilovich2009wikipedia,ferrucci2010building}. The basic components of WordNet are groups of words. Each group, called a synset, describes a distinct concept. 
In WordNet, synsets are linked by conceptual-semantic and lexical relations such as super-subordinate relation and antonym. We model WordNet as an MRN with the synsets as entities and the links between synsets as relations.

Following \citet{bordes2013translation}, we use a subset of WordNet for analysis. The dataset contains 40,943 synsets and 18 types of relations. A triple $(i,j,k)$ is called valid if relation $k$ from entity $i$ to entity $j$ exists, i.e., $Y_{ijk} = 1$. All the other triples are called invalid triples. Among more than $3.0 \times 10 ^{10}$ possible triples in WordNet, only 151,442 triples are valid. We assume that 141,442 valid triples and the same proportion of invalid triples are observed. The goal of our analysis is to recover the unobserved part of the knowledge base. We adopt the ranking procedure, which is commonly used in knowledge graph embedding literature, to evaluate link predictions. Given a valid triple $\lambda = (i,j,k)$, we rank estimated scores for all the invalid triples inside $\Lambda_{\cdot jk} = \left\{(i', j, k) \mid i' \in [N]\right\}$ in descending order and call the rank of $\phi\left(\hat{\bm x}_\lambda \right)$ as the head rank of $\lambda$, denoted by $H_{\lambda}$. Similarly, we can define the tail rank $T_{\lambda}$ and the relation rank $R_{\lambda}$ by ranking $\phi\left(\hat{\bm x}_{\lambda} \right)$ among the estimated scores of invalid triples in  $\Lambda_{ij\cdot}$ and $\Lambda_{i\cdot k}$, respectively. 
For a set $V$ of valid triples, the prediction performance can be evaluated by rank-based criteria, mean rank (MR), mean reciprocal rank (MRR), and hits at $q$ (Hits@q), which are defined as
$$\text{MR}_\text{E} = \frac{1}{2|V|}\sum_{\lambda \in V} H_{\lambda} + T_{\lambda}, ~~ \text{MR}_\text{R} = \frac{1}{|V|}\sum_{\lambda\in V} R_{\lambda},$$
$$\text{MRR}_\text{E} = \frac{1}{2|V|}\sum_{\lambda \in V} \frac{1}{H_{\lambda}} + \frac{1}{T_{\lambda}},~~\text{MRR}_\text{R} = \frac{1}{|V|}\sum_{\lambda \in V} \frac{1}{R_{\lambda}},$$ and 
$$\text{Hits}_\text{E}@q = \frac{1}{2|V|}\sum_{\lambda \in V} \bm 1_{\{H_\lambda\leq q\}} + \bm 1_{\{T_\lambda \leq q\}},~~\text{Hits}_\text{R}@q = \frac{1}{|V|}\sum_{\lambda \in V} \bm 1_{\{R_\lambda \leq q\}}.$$ The subscripts E and R represent the criteria for predicting entities and relations, respectively. Models with higher MRRs, Hits$@q$'s or lower MRs are more preferable. In addition, MRR is more robust to outliers than MR.

The three models described in \eqref{eq:phi_1}, \eqref{eq:phi_2}, and \eqref{eq:phi_3} are considered in our data analysis and we refer to them as Model 1, 2 and 3, respectively. For each model, the latent dimension $d$ takes value from $\{50, 100, 150, 200, 250\}$. Due to the high dimensionality of the parameter space, $L_2$ penalized MLE is used to obtain the estimated latent attributes $\hat{\bm x}$, with tuning parameters  $\rho_1 = 0$ and $\rho_2$ chosen from $\{0, 10^{-2}, 10^{-3}, 10^{-4}, 10^{-5}\}$ in \eqref{eq:pen_loglik}. Since information criteria based dimension and tuning parameter selection is computationally intensive for dataset of this scale, we set aside 5,000 of the unobserved valid triples as a validation set and select the $d$ and $\rho_2$ that produce the smallest $\text{MRR}_\text{E}$ on this validation set. The model with the selected $d$ and $\rho_2$ is then evaluated on the test set consisting of the rest 5,000 unobserved valid triples. 



The computed evaluation criteria on the test set are listed in Table \ref{table:WN_results}. The table also includes the selected $d$ and $\rho_2$ for each of the three score models. Models 2 and 3 generate similar performance. The MRRs for the two models are very close to 1, and the $\text{Hits}@q$'s are higher than 90\%, suggesting that the two models can identify the valid triples very well. Although Model 1 is inferior to the other two models in terms of most of the criteria, it outperforms them in $\text{MR}_\text{E}$. The results imply that Model 2 and Model 3 could perform extremely bad for a few triples. 

In addition to Models 1--3, we also display the performance of the Canonical Polyadic (CP) decomposition \cite{hitchcock1927cp} and a tensor factorization approach, RESCAL \cite{nickel2011rescal}. Their $\text{MRR}_\text{E}$ and $ \text{Hits}_\text{E}@10$ results on the WordNet dataset are extracted from \cite{trouillon2016complex}  and \cite{nickel2016holographic}, respectively. Both methods, especially CP, are outperformed by Model 3.

\begin{table}[htb]
	\centering
	\caption{Results for WordNet data analysis. The results for CP and RESCAL are extracted from \cite{trouillon2016complex}  and \cite{nickel2016holographic}.}\label{table:WN_results}
	\begin{tabular}{cccccccc}
		\hline
		$\text{Method}$ & $(d, \rho_2)$ & $\text{MR}_\text{E}$ & $\text{MRR}_\text{E}$ & $\text{Hits}_\text{E}@10$ & $\text{MR}_\text{R}$ & $\text{MRR}_\text{R}$ &  $\text{Hits}_\text{R}@1$ \\
		\hline
		Model 1 & $(100, 10^{-5})$ & 385 & 0.64 & 0.888 & 1.41 & 0.896 & 0.817 \\
		Model 2 & $(250, 10^{-4})$ & 769 & 0.94 & 0.945 & 1.31 & 0.968 & 0.959 \\
		Model 3 & $(200, 10^{-4})$ & 499 & 0.94 & 0.947 & 1.13 & 0.978 & 0.967 \\
		CP & -  & - & 0.075 & 0.125 & - & - & - \\
		RESCAL & - & - & 0.890 & 0.928 & - & - & - \\
		\hline
	\end{tabular}
\end{table}

\section{Concluding remarks}\label{sec:discussion}
In this article, we focused on the recovery of large-scale MRNs with a small portion of observations. We studied a generalized latent space model where entities and relations are associated with latent attribute vectors and conducted statistical analysis on the error of recovery. 
MLEs and pMLEs over a compact space are considered to estimate the latent attributes and the edge probabilities. We established non-asymptotic upper bounds for estimation error in terms of tail probability and risk, based on which we then studied the asymptotic properties when the size of MRN and latent dimension go to infinity simultaneously. A matching lower bound up to a log factor is also provided.

We kept $\phi$ generic for theoretical development. The choice of $\phi$ is usually problem-specific in practice. How to develop a data-driven method for selecting an appropriate $\phi$ is an interesting problem to investigate in future works.

Besides the latent space models, sparsity \citep{tran2020gm} or clustering assumptions \citep{jung2019tv} have been used to impose low-dimensional structures in single-relational networks. An MRN can be seen as a combination of several heterogeneous single-relational networks. The distribution of edges may vary dramatically across relations. Therefore, it is challenging to impose appropriate sparsity or cluster structures on MRNs. More empirical and theoretical studies are needed to quantify the impact of heterogeneous relations and to incorporate the information for recovering MRNs.

\appendix
\section*{Appendix}
\begin{proof}[Proof of Lemma \ref{lemma:upper:tail}]
	Let $\Theta_t = \left\{\bm x \in \Theta : L\left(M(\bm x), M^*\right)\geq t\right\}$ and $f(\bm x) = l\left(\bm x; Y_{\mathcal{S}}\right) - l \left(\bm x^*; Y_{\mathcal{S}}\right) $ be the log likelihood ratio. Therefore, $f$ is a random field living on $\Theta$. By writing $f(\bm x)$, we omit the second argument. In explicit form, $f(\bm x) = \sum\limits_{\lambda \in \Lambda} Z_\lambda$, where 
	\begin{equation}
	Z_\lambda = 1_{\lambda \in \cS} \left[Y_\lambda \log \frac{M_\lambda(\bm x)}{M_\lambda^*} + \left(1-Y_\lambda\right) \log \frac{1-M_\lambda(\bm x)}{1-M_\lambda^*} \right].
	\end{equation}
	We have $E\left[Z_\lambda\right] = - \gamma D\left(M_\lambda^* || M_\lambda(\bm x)\right)$ and $|Z_\lambda| \leq C$.  It follows that $f$ has properties
		(i) $f(\bm x^*) = 0$,
		(ii) $f(\hat{\bm x}) \geq 0$,
		(iii) $E\left[f(\bm x)\right] = - n L\left(M(\bm x), M^*\right)$.
	Based on the definition of $\Theta_t$ and property (ii), we have
	\begin{equation}\label{eq: to_supreme}
	P\left(L(\hat M, M^*) \geq  t\right) = P\left(\hat{\bm x} \in \Theta_t\right) \leq P\left(\sup_{\bm x \in \Theta_t} f(\bm x) \geq 0\right).
	\end{equation}
	From property (iii), we get that 
	\begin{equation}\label{eq:bound_mean}
    E\left[f(\bm x)\right] \leq -nt, \quad \forall \bm x \in \Theta_t.
	\end{equation}
    According to Lemma \ref{variance_bound_lemma} in Appendix, when $C \geq 2$, the variance of $Z_\lambda$ is bounded by
	\begin{equation*}
	\var\left[Z_\lambda\right] = \gamma M_\lambda^*(1-M_\lambda^*)\left(\log \frac{M_\lambda}{1-M_\lambda} - \log \frac{M_\lambda^*}{1-  M_\lambda^*}\right)^2 \leq 
	2 \gamma C D\left(M_\lambda^* || M_\lambda\right).
	\end{equation*}
	It follows that 
	\begin{equation}\label{eq:bound_var}
	\var\left[f(\bm x)\right] = \sum_{\lambda \in \Lambda} \var\left[Z_\lambda\right] \leq 2 \gamma C  \sum_{\lambda \in \Lambda} D\left(M_\lambda^* || M_\lambda\right) = -2 C E\left[f(\bm x)\right].
	\end{equation}
	By Bennett's inequality, 
    \begin{equation}
    P\left(f(\bm x) \geq -s\right) \leq \exp\left\{\frac{s + E\left[f(\bm x)\right]}{C} h\left(- \frac{C\left[s + E\left[f(\bm x)\right]\right]}{\var\left[f(\bm x)\right]}\right) \right\},
    \end{equation}
	where $ 0 < s < nt$ and $h(u) = \left(1+\frac{1}{u}\right) \log \left(1+u\right) - 1$ is an increasing function for $u > 0$.\\
	Hence by bounds in \eqref{eq:bound_mean}\eqref{eq:bound_var},
	\begin{equation}\label{eq:Bennett}
	P\left(f(\bm x) \geq -s\right) \leq \exp\left\{-\frac{nt-s}{C} h\left(\frac{s + E\left[f(\bm x)\right]}{2E\left[f(\bm x)\right]}\right)\right\} \leq \exp\left\{-\frac{nt-s}{C} h\left(\frac{1}{2} - \frac{s}{2nt}\right)\right\}.
	\end{equation}
	Let $\bm z = \argmax_{\bm x \in \Theta_t} f(\bm x)$ be the random vector  on $\Theta_t$ where $f(\bm x)$ reaches its maximum. Let $\mathcal{N}_{\epsilon, \mathcal{E}}$ and $\mathcal{N}_{\epsilon, \mathcal{R}}$ be the $\epsilon$-covering centers for $\mathcal{E}$ and  $\mathcal{R}$ respectively. Since $\mathcal{E}$ and $\mathcal{R}$ are balls of radius $U$, we can find $\epsilon$-coverings such that  $ | \mathcal{N}_{\epsilon, \mathcal{E}} | \leq \left(1+2U/\epsilon\right)^{d_E} $ and $ | \mathcal{N}_{\epsilon, \mathcal{R}} | \leq \left(1+2U/\epsilon\right)^{d_R}$. 
	For $\bm z = \left(\ttt_1, \ldots, \ttt_N, \w_1, \ldots, \w_K\right)$, there exists some $\bm x = \left(\ttt_1', \ldots, \ttt_N', \w_1', \ldots, \w_K'\right) \in  \mathcal{N}_{\epsilon, \mathcal{E}} ^N \times  \mathcal{N}_{\epsilon, \mathcal{R}} ^K$ such that $\| \ttt_i' - \ttt_i \| \leq \epsilon, \forall i \in [N]$ and $\| \w_k' - \w_k \| \leq \epsilon, \forall k \in [K].$ Therefore, 
	\begin{equation}\label{eq:covering_distance}
	f(\bm z) - f(\bm x) \leq  \sum_{\lambda \in \mathcal{S}}  | \phi(\bm z_\lambda) - \phi(\bm x_\lambda) | \leq \alpha \sum_{\lambda \in \mathcal{S}} \| \bm z_\lambda - \bm x_\lambda \| \leq \sqrt{3}  \alpha   |\cS| \epsilon.
	\end{equation}
	By Bennett's inequality, for every $\beta > 0$, 
	\begin{equation}\label{eq:Bennett_S}
	p\left(|\cS| - n > n\beta\right) \leq \exp\left\{-n\beta h\left(\frac{\beta}{1-\gamma}\right)\right\} \leq \exp\left\{-n\beta h(\beta)\right\}.
	\end{equation} 
	When $|\cS| \leq n(1+\beta)$, set $\epsilon = \frac{s}{\sqrt{3}\alpha n(1+\beta)} $, then $f(\bm z) - f(\bm x) \leq s$. 
	Combining \eqref{eq: to_supreme} \eqref{eq:Bennett}  and \eqref{eq:Bennett_S}, we get that
	\begin{equation}\label{eq:upper_bound_combine_steps}
	\begin{split}
	P\left(L(\hat M, M^*) \geq  t\right) & \leq P\left(\sup_{\bm x \in \Theta_t} f(\bm x) \geq 0, |\cS| \leq n(1+\beta) \right) + P\left(|\cS| > n(1+\beta)\right) \\
	&\leq P\left(\max\limits_{\bm x \in \mathcal{N}_{\epsilon, \mathcal{E}} ^N \times  \mathcal{N}_{\epsilon, \mathcal{R}} ^K } f(\bm x)\geq -s, |\cS| \leq n(1+\beta)\right) + P\left(|\cS| > n(1+\beta) \right)\\
	&\leq |\mathcal{N}_{\epsilon, \mathcal{E}} ^N \times  \mathcal{N}_{\epsilon, \mathcal{R}} ^K| \max\limits_{\bm x \in \mathcal{N}_{\epsilon, \mathcal{E}} ^N \times  \mathcal{N}_{\epsilon, \mathcal{R}} ^K} P\left(f(\bm x)\geq -s\right) + \exp\left\{-n\beta h(\beta)\right\}\\
	& \leq   \exp\left\{ -\frac{nt-s}{C} h\left(\frac{1}{2} - \frac{s}{2nt}\right) \right\} \left(1+\frac{2\sqrt{3}\alpha Un(1+\beta)}{s}\right)^m + \exp\left\{-n\beta h(\beta)\right\},
	\end{split}
	\end{equation}
	where $m = Nd_E + Kd_R$ is the degree of freedom. 
\end{proof}

\begin{proof}[Proof of Lemma \ref{lemma:upper:risk}]
	To bound $E\left[L(\hat M, M^*)\right]$, set $s = \frac{1}{2} nt$ and $\beta = 1+t$ in \eqref{eq:lemma:upper:tail} to get
	\begin{equation}
	P\left(L(\hat M, M^*) \geq t\right) \leq  \exp\left\{ -\frac{nt}{C_1}\right\} \left(1+ \frac{C_2}{2} + \frac{C_2}{t}\right)^m + \exp\left\{-\frac{1}{3}n(1+t)\right\}.
	\end{equation}
	By Fubini's Theorem,
	\begin{equation}\label{eq:mean_upper_bound_Fubini}
	E\left[L(\hat M, M^*)\right] = \int_{0}^{\infty}P\left(L(\hat M, M^*) \geq  t\right) dt \leq 
	t_0 + \int_{t_0}^{\infty}P\left(L(\hat M, M^*) \geq  t\right) dt.
	\end{equation}
	Let $C_3 = 2 \max\left[\{C_1, C_2\}\right]$ and $t_0 = C_3 \frac{m}{n} \log \frac{n}{m}$. When $t  \geq t_0$ and $\frac{n}{m} \geq C_2 + e $,
	\begin{equation}
	1 + \frac{C_2}{2} + \frac{C_2}{t} \leq 1 + \frac{C_2}{2} + \frac{C_2 n}{C_3m\log\frac{n}{m}} \leq  1 + \frac{C_2}{2} +\frac{n}{2m} \leq \frac{n}{m}.
	\end{equation}
	Thus
	\begin{equation} \label{eq:mean_upper_bound_integrand}
	P\left(L(\hat M, M^*) \geq  t\right) \leq \exp \left\{-\frac{nt}{C_1} + m \log \frac{n}{m} \right\} + \exp\left\{-\frac{1}{3}n(1+t)\right\}, \quad t \geq t_0.
	\end{equation}
	Hence by \eqref{eq:mean_upper_bound_Fubini} and \eqref{eq:mean_upper_bound_integrand}, 
	\begin{equation}
	\begin{split}
	E\left[L(\hat M, M^*)\right] & \leq t_0 + \frac{C_1}{n} \exp\left\{-\frac{nt_0}{C_1} + m \log \frac{n}{m} \right\} + \frac{3}{n}\exp\left\{-\frac{1}{3}n(1+t_0)\right\}\\
	& \leq C_3 \frac{m}{n} \log \frac{n}{m} + \frac{C_1}{n}\exp\left\{- m \log \frac{n}{m}\right\} + \frac{3}{n}\exp\left\{-\frac{1}{3}\left( n+ C_3 m \log \frac{n}{m}\right)\right\}. 
	\end{split}
	\end{equation}
\end{proof}

\begin{proof}[Proof of Theorem \ref{thm:upper}]
	When $t$ is a constant, let $s$ be absolute constant and $\beta = m \rightarrow \infty$ in Lemma \ref{lemma:upper:tail}. We analyze the order of three exponential terms on the right side of $\eqref{eq:lemma:upper:tail}$,
	\begin{gather*}
	-\frac{nt-s}{C} h\left(\frac{1}{2} - \frac{s}{2nt}\right) \sim  -\frac{h\left(\frac{1}{2}\right) }{C} nt, \\
	m \log\left(1+\frac{2\sqrt{3}\alpha Un(1+\beta)}{s}\right) \sim m \log(mn),\\
	-n\beta h(\beta) \sim -n m\log m.
	\end{gather*}
	Hence, both the second and the third term is asymptotically ignorable compared to the first term. It follows that 
	\begin{equation*}
	\log P\left(L(\hat M, M^*) \geq t\right) \lesssim - \frac{h\left(\frac{1}{2}\right)}{C}nt.
	\end{equation*}
	When $t = \frac{2C}{h\left(\frac{1}{2}\right)} \frac{m}{n} \log \frac{n}{m}$, let $s = m$ and $\beta$ be absolute constant. The exponential terms
	\begin{gather*}
	-\frac{nt-s}{C} h\left(\frac{1}{2} - \frac{s}{2nt}\right) \sim  - 2 m \log \frac{n}{m}, \\
	m \log\left(1+\frac{2\sqrt{3}\alpha Un(1+\beta)}{s}\right) =  m \log \frac{n}{m} + O(m).
	\end{gather*}
	The third term $\exp\left\{-n\beta h(\beta)\right\}$ is negligible. Therefore,
	\begin{equation}
	\log P\left(L(\hat M, M^*) \geq t\right) \lesssim -  m \log \frac{n}{m}.
	\end{equation}
	To bound the risk, we use similar approach as proof of Lemma $\ref{lemma:upper:risk}$. Let $s = m$, $\beta = 1 + t$ and $t_0 = \frac{2C}{h\left(\frac{1}{2}\right)} \frac{m}{n} \log \frac{n}{m}$.
	\begin{equation*}
	\begin{split}
	\int_{t_0}^{\infty} \exp\left\{	-\frac{nt-s}{C} h\left(\frac{1}{2} - \frac{s}{2nt}\right) \right\} dt & \leq \frac{C}{nh\left(\frac{1}{2} - \frac{s}{2nt_0}\right)} \exp\left\{	-\frac{nt_0-s}{C} h\left(\frac{1}{2} - \frac{s}{2nt_0}\right) \right\} \\
	& \sim \frac{C}{nh\left(\frac{1}{2}\right)} \exp\left\{ - 2m \log \frac{n}{m}\right\},
	\end{split}
	\end{equation*}
	\begin{equation*}
	m \log\left(1+\frac{2\sqrt{3}\alpha Un(1+\beta)}{s}\right) \leq  m \log\left(1+\frac{2\sqrt{3}\alpha Un(2 + t_0)}{m}\right) \sim m \log \frac{n}{m},
	\end{equation*}
	and 
	\begin{equation*}
	\int_{t_0}^{\infty} \exp\left\{-n (1+t) h\left(1+t\right)\right\} dt \leq \frac{3}{n} \exp\left\{-\frac{1}{3}n(1+t_0)\right\} = o\left(\exp\left\{ - m \log \frac{n}{m}\right\}\right).
	\end{equation*}
	It follows that
	\begin{equation}
	\begin{split}
	E\left[L(\hat M, M^*)\right] & \leq t_0 + \int_{t_0}^{\infty}P\left(L(\hat M, M^*) \geq  t\right) dt \\
	& \lesssim t_0 + o\left(t_0\right) \sim \frac{2C}{h\left(\frac{1}{2}\right)} \frac{m}{n} \log \frac{n}{m}.
	\end{split}
	\end{equation}	
	Since $h(\frac{1}{2}) \geq \frac{1}{5}$, we proof the results.
\end{proof}

\begin{lemma}\label{variance_bound_lemma}
	$\forall x,y \in [-C, C]$, we have
	\begin{equation}
	\sigma(x)\left(1-\sigma(x)\right) \left(y-x\right)^2 \leq 2 \max\left\{C,2\right\} D\left(\sigma(x)||\sigma(y)\right),
	\end{equation}
\end{lemma}
\begin{proof}
	We only need to show the result for $x \geq 0$ by symmetry. For any fixed $x \in [0, C]$, define  $g(y) = 2 C_m D\left(\sigma(x)||\sigma(y)\right) - \sigma(x)\left(1-\sigma(x)\right) \left(y-x\right)^2 $, where $C_m = \max\left\{C,2\right\}$.  Since
	\begin{equation}
	g'(y) = 2C_m (\sigma(y) - \sigma(x))-2\sigma(x)(1-\sigma(x))(y-x),
	\end{equation}
	we have $g'(x) = g(x) = 0$. It remains to show that $\frac{g'(y)}{y-x}>0$ for all $y \in [-C,C]\setminus\left\{x\right\}$, then $g(x)$ reaches the minimum at $x=0$ and $g(y) \geq0$ on $[-C,C]$. Equivalently, we want to show that 
	$$C_m (\sigma(y) - \sigma(x)) / (y-x) > \sigma(x)(1-\sigma(x)).$$
	Note that $ (\sigma(y) - \sigma(x)) / (y-x) $ is the slope of secant line on logistic function and reaches its minimum at $y=C$.
	It suffices to show that 
	\begin{equation}
	(C-x)\sigma(x) (1-\sigma(x)) + C_m\sigma(x) \leq C_m\sigma(C), \forall x \in [0,C]
	\end{equation}
	Let $h(x)$ be left side above. By taking the derivative, we get
	$$h'(x) = \left[C_m-1 - (C-x)\left(2\sigma(x)-1\right)\right] \sigma(x)\left(1- \sigma(x)\right).$$
	If $1 \leq x \leq C$, then $(C-x)\left(2\sigma(x)-1\right)\leq C-1 \leq C_m-1$. If $ 0 \leq x \leq 1$, then $(C-x)\left(2\sigma(x) -1\right) \leq C\left(2\sigma(1)-1\right) \leq \frac{1}{2}C \leq C_m-1$. Therefore, $h'(x) \geq 0$ on $[0,C]$. It follows that $h(x) \leq h(C) = C_m \sigma(C)$.
\end{proof}

To prove the lower bound in Theorem \ref{thm:lower}, we will use Lemma \ref{lemma:GV_bound} -- \ref{lemma:KL_upper}. Since Lemma \ref{lemma:GV_bound} \cite{massart2007concentration} and Lemma \ref{lemma:Fano} \cite{cover2006info} are well established results in literature, we will skip the proofs.
\begin{lemma}[Gilbert-Varshamov bound]\label{lemma:GV_bound}
	There exists a subset $\mathcal{V}$ of the $d$-dimensional hypercube $\left\{-1,1\right\}^d $ of size at least $\exp\{d/8\}$ such that the Hamming distance
	\begin{equation}
	\sum_{i = 1 }^d 1_{\bm u_i \neq \bm v_i} \geq \frac{1}{4}d
	\end{equation}
	for all $\bm u \neq \bm v$ with $\bm u, \bm v \in \mathcal{V}$.
\end{lemma}

\begin{lemma}[Fano's inequality]\label{lemma:Fano}
	Let $V$ be a uniform random variable taking values in a finite set $\cV$ with cardinality $|\cV| \geq 2$. For any Markov chain $V \rightarrow X \rightarrow \hat V$, 
	\begin{equation}
	P\left(\hat V \neq V\right) \geq 1 - \frac{I(V;X) + \log 2}{\log\left(|\cV|\right)},
	\end{equation}
	where $I(V; X)$ is the mutual information between $V$ and $X$.
\end{lemma}

\begin{lemma} \label{lemma:KL_upper}
	Suppose that $p,q \in (0,1)$. Then
	\begin{equation}
	D(p||q) \leq \frac{(p-q)^2}{q(1-q)}.
	\end{equation}
\end{lemma}

\begin{proof}
	Since $D(1-p || 1-q) = D(p || q)$, it suffices to show for case $p \leq q$. View $D(p||q)$ as a function of $q$. By mean value theorem, there exists $ \xi \in [p,q]$ such that 
	\begin{equation}
	D(p || q) - D(p || p) = \frac{\xi-p}{\xi(1-\xi)} (q-p)
	\end{equation} 
	Note that $\frac{\xi-p}{\xi(1-\xi)} $ is increasing in $\xi$ and $D(p||p)=0$. Hence, $D(p||q) \leq \frac{(q-p)^2}{q(1-q)}$.
\end{proof}

\begin{proof}[Proof of Theorem \ref{thm:lower}]
	Let $\bm u_0 = (\ttt_0, \ttt_0', \w_0)$, 
	$\tilde{\bm x} = (\underbrace{\ttt_0, \ldots, \ttt_0}_{\lfloor\frac{N}{2}\rfloor}, \underbrace{\ttt_0', \ldots, \ttt_0'}_{\lceil\frac{N}{2}\rceil}, \underbrace{\w_0, \ldots, \w_0)}_K$ and 
	$$\tilde{\Lambda} = \left\{(i,j,k)\in \Lambda \mid i \leq \lfloor\frac{N}{2}\rfloor, j > \lfloor\frac{N}{2}\rfloor\right\} \subset \Lambda$$ 
	with cardinality $| \tilde{\Lambda} | = \lfloor\frac{N}{2}\rfloor \lceil\frac{N}{2}\rceil K $. 
	If $\bm x \in \mathcal{N}_r(\tilde{\bm x})$, then $\bm x_\lambda \in \mathcal{N}_r(\bm u_0)$ for every $\lambda \in \tilde{\Lambda}$. Hence according to Assumption \ref{assumption:non_constant}, 
	 \begin{equation}\label{eq:apply_assumption_non_constant}
	 \left| \sigma\left(\phi(\bm x_\lambda)\right) - \sigma\left(\phi(\bm x_\lambda')\right) \right| \geq \kappa \|  \bm x_\lambda - \bm x_\lambda' \|, \quad \forall \bm x, \bm x' \in \mathcal{N}_r(\tilde{\bm x}), \lambda \in \tilde{\Lambda}.
	 \end{equation}
	 We will find $\bm x^*$ in the vicinity of  $\tilde{\bm x}$ such that  \eqref{eq: minimax} holds.
	 
	Let $\cH_E = \left\{-\delta/\sqrt{d_E}, \delta/\sqrt{d_E}\right\}^{Nd_E}$ and $\cH_R = \left\{-\delta/\sqrt{d_R}, \delta/\sqrt{d_R}\right\}^{Kd_R}$ be two hypercubes. According to Gilbert-Varshamov bound in Lemma \ref{lemma:GV_bound}, there exist $\cV_E \subset \cH_E$ and $\cV_R \subset \cH_R$ such that $|\cV_E| \geq \exp\left\{Nd_E/8\right\}$, $|\cV_R| \geq \exp\left\{Kd_R/8\right\}$ and 
	\begin{equation} \label{eq:apply_GV_entity}
	\sum_{i = 1 }^{Nd_E} 1_{\bm u_i \neq \bm v_i} \geq \frac{1}{4}Nd_E, \quad \forall \bm u, \bm v \in \cV_E, \bm u \neq \bm v,
	\end{equation}
	\begin{equation} \label{eq:apply_GV_relation}
	\sum_{i = 1 }^{Kd_R} 1_{\bm u_i \neq \bm v_i} \geq \frac{1}{4}Kd_R, \quad \forall \bm u, \bm v \in \cV_R, \bm u \neq \bm v.
	\end{equation}
	For $\bm u = (\ttt_1, \ldots, \ttt_N) \in \cV_E$, $\bm v = (\ttt_1', \ldots, \ttt_N') \in \cV_E$ and $\bm u \neq \bm v$,  \eqref{eq:apply_GV_entity} suggests that 
	\begin{equation}
	\sum_{i=1}^{N} \|\ttt_i - \ttt_i' \|^2  \geq \sum_{i=1}^{N} \left(2\delta/\sqrt{d_E}\right)^2 \frac{1}{4} N d_E = N \delta^2,
	\end{equation}
	Likewise, 
	from \eqref{eq:apply_GV_relation} we can get that 
	\begin{equation}
	\sum_{i=1}^{K} \|\w_k - \w_k' \| \geq K \delta^2,
	\end{equation}
	with $\bm u = (\w_1, \ldots, \w_K) \in \cV_R$, $\bm v = (\w_1', \ldots, \w_K') \in \cV_R$ and $\bm u \neq \bm v$.

    Let $\cV = \left\{\tilde{\bm x} + \bm e \mid \bm e \in \cV_E \times \cV_R\right\} = \{\bm x^{(1)}, \ldots, \bm x^{(T)}\} $ where $T = | \cV_E | | \cV_R| \geq \exp\left\{m/8\right\}$. By the definition of $\delta$-neighborhood and size of hypercubes, we have $\cV \subset \mathcal{N}_\delta\left(\tilde{\bm x}\right)$ and thus property in \eqref{eq:apply_assumption_non_constant} holds for $\delta \leq r$. The corresponding tensors are denoted as $M(\cV) = \left\{M^{(1)}, \ldots, M^{(T)}\right\}$ where $M^{(i)} = M\left(\bm x^{(i)}\right)$ for $i \in [T]$.
	Let $\bm z = \argmin\limits_{\bm x \in \cV} \| \hat M - M(\bm x) \|$, thus $M(\bm z)$ is the closet tensor to $\hat{M}$ in $M(\cV)$ under Frobenius norm.  By triangular inequality, 
	\begin{equation}
	\| \hat M -  M^{(i)}  \| \geq \frac{1}{2}\left( \| \hat M - M^{(i)} \|+ \| \hat M -  M(\bm z \|) \right)  \geq \frac{1}{2} \| M^{(i)} - M(\bm z) \|, \quad \forall i \in [T].
	\end{equation}
	Note that $\bm z, \bm x^{(i)} \in \cV$, according to Pinsker's inequality and \eqref{eq:apply_assumption_non_constant},
	\begin{equation*}
	L\left(\hat M, M^{(i)}\right) \geq \frac{2}{|\Lambda|} \| \hat M - M^{(i)} \|^2 \geq \frac{1}{2|\Lambda|} \| M^{(i)} - M(\bm z)\|^2 \geq \frac{\kappa^2}{2|\Lambda|} \sum_{\lambda \in \tilde \Lambda} \| \bm x^{(i)}_\lambda - \bm z_\lambda \|^2.
	\end{equation*}
	For all $\bm x \neq \bm x'$ with $\bm x, \bm x' \in \cV$ and $N \geq 2$,
	\begin{equation}\label{eq:packing_dist}
	\begin{split}
	\frac{1}{| \Lambda |} \sum_{\lambda \in \tilde \Lambda} \| \bm x_\lambda - \bm x_\lambda' \|^2 &\geq \frac{1}{|\Lambda|} \left( \lfloor \frac{N}{2} \rfloor K \sum_{i \in [N]}  \| \ttt_i - \ttt_i' \|^2 + \lfloor\frac{N}{2}\rfloor \lceil\frac{N}{2}\rceil \sum_{k \in [K]} \|\w_k - \w_k' \|^2 \right)\\
	&\geq \min\left\{\frac{1}{3}	\frac{1}{N}\sum_{i \in [N]} \| \ttt_i - \ttt_i'  \|^2, \frac{2}{9} \frac{1}{K}\sum_{k \in [K]} \|\w_k - \w_k' \|^2 \right\} = \frac{2}{9} \delta^2.
	\end{split}
	\end{equation}
	Hence when $\bm x^{(i)} \neq \bm z$,
	\begin{equation}
	  L\left(\hat M, M^{(i)}\right) \geq  \frac{1}{9} \kappa^2 \delta^2.
	\end{equation}
	Let $P_i$ denote the probability measure under $\bm x^{(i)}$. Results above show that
	\begin{equation}\label{eq:lower:bound_by_z}
	P_i\left(L(\hat M, M^{(i)})  \geq  \frac{1}{9} \kappa^2 \delta^2 \right)  \geq P_i\left(\bm x^{(i)} \neq \bm z\right), \quad \forall i \in [N].
	\end{equation}
	Assign a prior on $\bm x$ that is uniform on $\cV$ and denote by $P_\cV$ the Bayes average probability with respect to the prior. By Fano's inequality in Lemma \ref{lemma:Fano}, 
	\begin{equation}\label{eq: Fano}
	P_\cV\left(\bm z \neq  \bm x\right) \geq 1 - \frac{I(\bm x; Y_{\mathcal{S}}) + \log 2}{\log |T|},
	\end{equation}
	where $I(\bm x; X_{\mathcal{S}})$ is the mutual information between $\bm x$ and $Y_{\mathcal{S}}$. It can be bounded by the maximum pairwise KL divergence of $Y_\mathcal{S}$ under $P_i$ and $P_j$ as follows,
	\begin{equation}
	\begin{split}
	I(\bm x, Y_{\mathcal{S}}) = & \frac{1}{T} \sum_{i=1}^T D\left(P_i(Y_\cS) || P_\cV(Y_\cS)\right) \leq \max\limits_{i \neq j} D\left(P_i(Y_\cS) || P_j(Y_\cS)\right) = \\
	&\max\limits_{i \neq j} \sum_{\lambda \in \Lambda} D\left(P_i(Y_\lambda, \lambda \in \cS) || P_j(Y_\lambda, \lambda \in \cS)\right) = \max\limits_{i \neq j} n L\left(M^{(i)}, M^{(j)}\right).
	\end{split}
	\end{equation}
	Since $\sigma(\cdot)$ is logistic function, the derivative $\sigma'(x) = \sigma(x) \left(1-\sigma(x)\right) < 1$. By Assumption \ref{assumption:lip}, $\phi(\cdot)$ is Lipschitz continuous with coefficient $\alpha$ , we get that $\sigma(\phi(\cdot))$ is also Lipschitz continuous with coefficient $\alpha$. 
	Let  $b = \sup\limits_{\bm u \in \mathcal{N}_r(\bm u_0)} \sigma\left(\phi(\bm u)\right)$, by Lemma \ref{lemma:KL_upper} we get
	\begin{equation}
	L(M^{(i)}, M^{(j)}) \leq \frac{\| M^{(i)} - M^{(j)} \|^2}{|\Lambda|b(1-b)} \leq \frac{\alpha^2 \sum_{\lambda \in \Lambda} \|  \bm x^{(i)}_\lambda - \bm x^{(j)}_\lambda \| ^2}{|\Lambda|b(1-b)} \leq \frac{ 3(2 \delta)^2 \alpha^2}{b(1-b)} = \frac{12\alpha^2\delta^2}{b(1-b)}
	\end{equation}
	for all $i, j \in [N]$. Hence, there exists $\bm x^{(i)} \in \cV$ such that 
	\begin{equation}
		P_i\left( \bm z \neq \bm x^{(i)}\right) \geq 1 - \frac{\frac{12\alpha^2\delta^2 n}{b(1-b)} + \log 2}{\log |T|} \geq 1 - \frac{\frac{12\alpha^2\delta^2 n}{b(1-b)} + 1}{m / 8}.
	\end{equation}
	Let $\bm x^* = \bm x^{(i)}$ , $P = P_i$ and 
	\begin{equation*}
	\delta^2 = \frac{(m/16-1) b(1-b)}{12\alpha^2n} \leq r^2.
	\end{equation*}
	It follows from \eqref{eq:lower:bound_by_z} that 
	\begin{equation}
	P\left( L(\hat M, M^{(i)})  \geq  \frac{\kappa^2 b(1-b)}{108\alpha^2} \frac{m/16-1}{n} \right) \geq \frac{1}{2}.
	\end{equation}
\end{proof}

\begin{proof}[Proof of Theorem \ref{thm:extension:pen}]
	We will show the result by continuing the proof of Lemma \ref{lemma:upper:tail} and Theorem \ref{thm:upper} with some modifications. Let $f_\rho(\bm x)$ be the penalized log likelihood ratio, we have 
	\begin{equation}
	\begin{split}
	f_\rho(\bm x) &= l_\rho\left(\bm x; Y_\cS\right) - l_\rho\left(\bm x^*; Y_\cS\right) \\
	&= f(\bm x) - \rho_1\left(\|\bm x\|_1 - \|\bm x^*\|_1\right)  - \rho_2\left(\|\bm x\|^2 - \|\bm x^*\|^2\right) \\
	&\leq f(\bm x) + \sqrt{2}\rho_1 (N+K) U + \rho_2 (N+K)U^2
	\end{split}
	\end{equation}
	According to \eqref{eq:covering_distance}, there exists $\bm x$ among the $\epsilon$-covering centers such that 
	\begin{equation}
	\begin{split}
	f_\rho(\bm z) - f_\rho(\bm x) &= f(\bm z) - f(\bm x) - \rho_1\left(\|\bm z\|_1 - \| \bm x\|_1\right) - \rho_2\left(\|\bm z\|^2 - \| \bm x\|^2\right) \\
	&\leq \sqrt{3} \alpha |\cS|\epsilon +  \sqrt{2} \rho_1 (N+K)  \epsilon + 2 \rho_2 (N+K) U \epsilon,
	\end{split}
	\end{equation}
	where $\bm z = \argmax_{\bm x \in \Theta_t} f_\rho(\bm x)$.
	It follow that when $|\cS| \leq n(1+\beta)$ and $f_\rho(\bm z) \geq 0$,
	\begin{equation}
	\begin{split}
	f_\rho(\bm x) & \geq  - \sqrt{3} \alpha |\cS|\epsilon - \sqrt{2} \rho_1 (N+K)  \epsilon - 2 \rho_2 (N+K) U \epsilon \\
	& \geq -s - \frac{(N+K) s}{\alpha n (1+\beta)} \left(\sqrt{\frac{2}{3}} \rho_1 + \frac{2}{\sqrt{3}} \rho_2 U\right),
	\end{split}
	\end{equation}
	with $\epsilon = \frac{s}{\sqrt{3}\alpha n(1+\beta)} $.
	Hence, we can rewrite \eqref{eq:upper_bound_combine_steps} as 
	\begin{equation}
	\begin{split}
	P\left(L(\hat M, M^*) \geq  t\right) & \leq P\left(\sup_{\bm x \in \Theta_t} f_\rho(\bm x) \geq 0, |\cS| \leq n(1+\beta) \right) + P\left(|\cS| > n(1+\beta)\right) \\
	&\leq |\mathcal{N}_{\epsilon, \mathcal{E}} ^N \times  \mathcal{N}_{\epsilon, \mathcal{R}} ^K| P\left(f(\bm x)\geq -s_\rho\right) + \exp\left\{-n\beta h(\beta)\right\}\\
	& \leq   \exp\left\{ -\frac{nt-s_\rho}{C} h\left(\frac{1}{2} - \frac{s_\rho}{2nt}\right) \right\} \left(1+\frac{2\sqrt{3}\alpha Un(1+\beta)}{s}\right)^m + \exp\left\{-n\beta h(\beta)\right\},
	\end{split}
	\end{equation}
	where $$s_\rho = s + \frac{(N+K) s}{\alpha n (1+\beta)} \left(\sqrt{\frac{2}{3}} \rho_1 + \frac{2}{\sqrt{3}} \rho_2 U\right)+  \sqrt{2}\rho_1 (N+K) U + \rho_2 (N+K)U^2.$$
	Therefore, $s_\rho = s + o(s) + O(N) = o(nt)$ when $t$ and $s$ are absolute constant or when $t = \frac{2C}{h\left(\frac{1}{2}\right)} \frac{m}{n} \log \frac{n}{m}$ and $s = m$. Hence the proof of Theorem \ref{thm:upper} applies and the asymptotic results hold.
\end{proof}

\begin{proof}[Proof of Corollary \ref{cor:mse:upper},  \ref{cor:mse:lower} and \ref{cor:err:upper}]
	To show these corollaries, we associate $MSE_\phi$ and $\widehat{err}$ with $L(\hat M, M^*)$.
	The first and second order derivatives of $ D\left(\sigma(x) || \sigma(y)\right)$ as a function of $y$ are
	\begin{equation}
	\frac{\partial}{\partial y} D\left(\sigma(x) || \sigma(y)\right) = \sigma(y) - \sigma(x), \quad \frac{\partial^2}{\partial^2 y}D\left(\sigma(x) || \sigma(y)\right)  = \sigma(y)\left(1- \sigma(y)\right).
	\end{equation}
	By Taylor expansion, there exists $\xi = u x + (1-u) y$ with $u \in (0,1)$ such that $D\left(\sigma(x) || \sigma(y)\right)  = \frac{1}{2} \sigma(\xi)\left(1- \sigma(\xi)\right) (y-x)^2$. Hence, for $x, y\in [-C, C]$,
	\begin{equation}
	\frac{1}{2} \sigma(C)\left(1-\sigma(C)\right) (y-x)^2\leq D\left(\sigma(x) || \sigma(y)\right) \leq \frac{1}{8} (y-x)^2  .
	\end{equation}
	It follows that 
	\begin{equation}\label{eq:cor:mse_sandwich}
	\frac{1}{2} \sigma(C)\left(1-\sigma(C)\right) MSE_\phi \leq L\left(\hat M, M^*\right) \leq \frac{1}{8} MSE_\phi.
	\end{equation}
	where $MSE_\phi = \frac{1}{|\Lambda|} \sum_{\lambda \in \Lambda} \left(\phi(\hat{\bm x}_\lambda) - \phi(\bm x_\lambda^*) \right)^2$ is the mean squared error of edge scores. The upper bound of $MSE_\phi$ follows from Theorem \ref{thm:extension:pen} and left half of \eqref{eq:cor:mse_sandwich}. By Theorem \ref{thm:lower} and right half of \eqref{eq:cor:mse_sandwich}, we get the corresponding lower bound. Likewise, for $\widehat{err}$ we can derive the upper bound by
	\begin{equation}
	L\left(\hat M, M^*\right) = \frac{1}{|\Lambda|} \sum_{\lambda \in \Lambda} D\left(M^*_\lambda || \hat M_\lambda\right) \geq 
	\frac{1}{|\Lambda|} \sum_{\lambda \in \Lambda} 1_{\hat Y_\lambda \neq Y_\lambda^*} D\left(\frac{1}{2} + \varepsilon || \frac{1}{2}\right) \geq 2 \epsilon^2 \widehat{err} .
	\end{equation}
\end{proof}

\begin{proof}[Proof of Theorem \ref{thm:sparse}]
	Let $\Theta_\tau = \left\{\bm x \in \cE^N \times \cR^K \mid \|\bm x\|_0 \leq m_\tau\right\}$ be subspaces of $\Theta$ with at most $m_\tau$ non-zeros and $\mathcal N_{\Theta_\tau}$ be its $\epsilon$-covering centers. There are ${m \choose m_\tau}$ combinations of support, and each subspace has a covering number of $\left(1 + \frac{2U}{\epsilon}\right)^{m_\tau}$. Hence, the overall $\epsilon$-covering number of $\Theta_\tau$ would be 
	\begin{equation}
	| \mathcal N_{\Theta_\tau} | = {m \choose m_\tau} \left(1 + \frac{2U}{\epsilon}\right)^{m_\tau}.
	\end{equation}
	We can rewrite Lemma \ref{lemma:upper:tail} as
	\begin{equation}
	P\left(L(\hat M, M^*) \geq  t\right) \leq  \exp\left\{ - \Rmnum{1} + \Rmnum{2} \right\}  + \exp\left\{-\Rmnum{3}\right\},
	\end{equation}
	where 
	\begin{gather*}
	\Rmnum{1} = \frac{nt-s}{C} h\left(\frac{1}{2} - \frac{s}{2nt} \right),\\
	\Rmnum{2} = \log {m \choose m_\tau} + m_\tau \log \left(1+\frac{2\sqrt{3}\alpha Un(1+\beta)}{s}\right),\\
	\Rmnum{3}= n\beta h(\beta).
	\end{gather*}
	By Stirling's approximation, 
	\begin{equation}
	\begin{split}
	\log {m \choose m_\tau}  & \sim - m_\tau  \log \tau - \left(m - m_\tau \right) \log(1-\tau) -\frac{1}{2} \log m \\
	& \lesssim  m_\tau  \left(-\log \tau+1\right) -\frac{1}{2}\log m = O(m_\tau).
	\end{split}
	\end{equation}
	To get the results, when $t$ is absolute constant, let $s$ be absolute constant and $\beta = m$. When $t = \frac{2C}{h\left(\frac{1}{2}\right)} \frac{m_\tau}{n} \log \frac{n}{m_\tau}$, let $s = m_\tau$ and $\beta$ be absolute constant. For risk upper bound, select $s = m_\tau, \beta = 1+t$ and $t_0 = \frac{2C}{h\left(\frac{1}{2}\right)} \frac{m_\tau}{n} \log \frac{n}{m_\tau}$. At last, use $h(\frac{1}{2}) \geq \frac{1}{5}$.
\end{proof}

\bibliographystyle{apalike}
\bibliography{mrg}

\begin{thebibliography}{}

\bibitem[Bordes et~al., 2013]{bordes2013translation}
Bordes, A., Usunier, N., Garcia-Duran, A., Weston, J., and Yakhnenko, O.
  (2013).
\newblock Translating embeddings for modeling multi-relational data.
\newblock In Burges, C. J.~C., Bottou, L., Welling, M., Ghahramani, Z., and
  Weinberger, K.~Q., editors, {\em Advances in Neural Information Processing
  Systems 26}, pages 2787--2795. Curran Associates, Inc.

\bibitem[Cover and Thomas, 2006]{cover2006info}
Cover, T.~M. and Thomas, J.~A. (2006).
\newblock {\em Elements of Information Theory, Second Edition.}
\newblock Wiley.

\bibitem[Duchi et~al., 2011]{duchi2011adaptive}
Duchi, J., Hazan, E., and Singer, Y. (2011).
\newblock Adaptive subgradient methods for online learning and stochastic
  optimization.
\newblock {\em Journal of Machine Learning Research}, 12(Jul):2121--2159.

\bibitem[Ferrucci et~al., 2010]{ferrucci2010building}
Ferrucci, D., Brown, E., Chu-Carroll, J., Fan, J., Gondek, D., Kalyanpur,
  A.~A., Lally, A., Murdock, J.~W., Nyberg, E., Prager, J., et~al. (2010).
\newblock Building watson: An overview of the deepqa project.
\newblock {\em AI magazine}, 31(3):59--79.

\bibitem[Gabrilovich and Markovitch, 2009]{gabrilovich2009wikipedia}
Gabrilovich, E. and Markovitch, S. (2009).
\newblock Wikipedia-based semantic interpretation for natural language
  processing.
\newblock {\em Journal of Artificial Intelligence Research}, 34:443--498.

\bibitem[Garcia-Duran et~al., 2016]{garcia2016combining}
Garcia-Duran, A., Bordes, A., Usunier, N., and Grandvalet, Y. (2016).
\newblock Combining two and three-way embedding models for link prediction in
  knowledge bases.
\newblock {\em Journal of Artificial Intelligence Research}, 55:715--742.

\bibitem[Hitchcock, 1927]{hitchcock1927cp}
Hitchcock, F.~L. (1927).
\newblock The expression of a tensor or a polyadic as a sum of products.
\newblock {\em Journal of Mathematics and Physics}, 6(1-4):164--189.

\bibitem[Hoff et~al., 2002]{hoff2002latent}
Hoff, P.~D., Raftery, A.~E., and Handcock, M.~S. (2002).
\newblock Latent space approaches to social network analysis.
\newblock {\em Journal of the American Statistical Association},
  97(460):1090--1098.

\bibitem[{Jung} et~al., 2019]{jung2019tv}
{Jung}, A., {Hero, III}, A.~O., {Mara}, A.~C., {Jahromi}, S., {Heimowitz}, A.,
  and {Eldar}, Y.~C. (2019).
\newblock Semi-supervised learning in network-structured data via total
  variation minimization.
\newblock {\em IEEE Transactions on Signal Processing}, 67(24):6256--6269.

\bibitem[Kanojia et~al., 2017]{kanojia2017enhancing}
Kanojia, V., Maeda, H., Togashi, R., and Fujita, S. (2017).
\newblock Enhancing knowledge graph embedding with probabilistic negative
  sampling.
\newblock In {\em Proceedings of the 26th International Conference on World
  Wide Web Companion}, pages 801--802.

\bibitem[Kotnis and Nastase, 2017]{kotnis2017analysis}
Kotnis, B. and Nastase, V. (2017).
\newblock Analysis of the impact of negative sampling on link prediction in
  knowledge graphs.
\newblock {\em arXiv preprint arXiv:1708.06816}.

\bibitem[Lin et~al., 2015]{lin2015learning}
Lin, Y., Liu, Z., Sun, M., Liu, Y., and Zhu, X. (2015).
\newblock Learning entity and relation embeddings for knowledge graph
  completion.
\newblock In {\em AAAI}, pages 2181--2187.

\bibitem[Liu et~al., 2017]{liu2017analogical}
Liu, H., Wu, Y., and Yang, Y. (2017).
\newblock Analogical inference for multi-relational embeddings.
\newblock In {\em Proceedings of the 34th International Conference on Machine
  Learning}, volume~70, pages 2168--2178.

\bibitem[Massart, 2007]{massart2007concentration}
Massart, P. (2007).
\newblock {\em Concentration inequalities and model selection}, volume 1896 of
  {\em Lecture notes in Mathematics}.
\newblock Springer.

\bibitem[McCray, 2003]{mccray2003upper}
McCray, A.~T. (2003).
\newblock An upper-level ontology for the biomedical domain.
\newblock {\em Comparative and Functional Genomics}, 4(1):80--84.

\bibitem[Miller, 1995]{miller1995wordnet}
Miller, G.~A. (1995).
\newblock Wordnet: a lexical database for english.
\newblock {\em Communications of the ACM}, 38(11):39--41.

\bibitem[Min et~al., 2013]{min2013distant}
Min, B., Grishman, R., Wan, L., Wang, C., and Gondek, D. (2013).
\newblock Distant supervision for relation extraction with an incomplete
  knowledge base.
\newblock In {\em HLT-NAACL}, pages 777--782.

\bibitem[Nickel et~al., 2016]{nickel2016holographic}
Nickel, M., Rosasco, L., Poggio, T.~A., et~al. (2016).
\newblock Holographic embeddings of knowledge graphs.
\newblock In {\em AAAI}, pages 1955--1961.

\bibitem[Nickel et~al., 2011]{nickel2011rescal}
Nickel, M., Tresp, V., and Kriegel, H.-P. (2011).
\newblock A three-way model for collective learning on multi-relational data.
\newblock In {\em Proceedings of the 28th International Conference on
  International Conference on Machine Learning}, ICML’11, page 809–816,
  Madison, WI, USA. Omnipress.

\bibitem[Robbins and Monro, 1951]{robbins1951sa}
Robbins, H. and Monro, S. (1951).
\newblock A stochastic approximation method.
\newblock {\em The Annals of Mathematical Statistics}, 22(3):400--407.

\bibitem[Scott and Matwin, 1999]{scott1999feature}
Scott, S. and Matwin, S. (1999).
\newblock Feature engineering for text classification.
\newblock In {\em ICML}, volume~99, pages 379--388.

\bibitem[Socher et~al., 2013]{socher2013reasoning}
Socher, R., Chen, D., Manning, C.~D., and Ng, A. (2013).
\newblock Reasoning with neural tensor networks for knowledge base completion.
\newblock In {\em Advances in neural information processing systems}, pages
  926--934.

\bibitem[{Tran} et~al., 2020]{tran2020gm}
{Tran}, N., {Abramenko}, O., and {Jung}, A. (2020).
\newblock On the sample complexity of graphical model selection from
  non-stationary samples.
\newblock {\em IEEE Transactions on Signal Processing}, 68:17--32.

\bibitem[Trouillon et~al., 2016]{trouillon2016complex}
Trouillon, T., Welbl, J., Riedel, S., Gaussier, {\'E}., and Bouchard, G.
  (2016).
\newblock Complex embeddings for simple link prediction.
\newblock In {\em International Conference on Machine Learning}, pages
  2071--2080.

\bibitem[Wang et~al., 2014]{wang2014knowledge}
Wang, Z., Zhang, J., Feng, J., and Chen, Z. (2014).
\newblock Knowledge graph embedding by translating on hyperplanes.
\newblock In {\em AAAI}, pages 1112--1119.

\bibitem[Yang et~al., 2015]{yang2015embedding}
Yang, B., Yih, S. W.-t., He, X., Gao, J., and Deng, L. (2015).
\newblock Embedding entities and relations for learning and inference in
  knowledge bases.
\newblock In {\em Proceedings of the International Conference on Learning
  Representations}.

\bibitem[Zou and Hastie, 2005]{zou2005regularization}
Zou, H. and Hastie, T. (2005).
\newblock Regularization and variable selection via the elastic net.
\newblock {\em Journal of the Royal Statistical Society B}, 67(2):301--320.

\end{thebibliography}

\end{document}